\documentclass[11pt,a4paper,reqno]{amsart}
\usepackage[english]{babel}
\usepackage{amssymb}
\usepackage{amsmath}
\usepackage{setspace}
\usepackage{amsthm}
\usepackage{float}
\usepackage{amscd} 
\usepackage{graphicx} 
\usepackage{caption}
\usepackage{subcaption}

 \theoremstyle{plain}
 \newtheorem{theorem}{\textbf{Theorem}}[section]
 \newtheorem{proposition}[theorem]{\textbf{Proposition}}
 \newtheorem{lemma}[theorem]{\textbf{Lemma}}
 
 \newtheorem*{conjecture}{Conjecture}
\theoremstyle{definition}
 \newtheorem{example}[theorem]{\textbf{Example}}
 \newtheorem{definition}[theorem]{\textbf{Definition}}
 \newtheorem{remark}[theorem]{\textbf{Remark}}
 \numberwithin{equation}{section}

\renewcommand{\leq}{\leqslant}
\renewcommand{\geq}{\geqslant}

\newcommand{\M}{\mathbb{M}}
\newcommand{\R}{\mathbb{R}}

\newcommand{\RP}{\mathbb{RP}}
\newcommand{\T}{\mathbb{T}}
\newcommand{\K}{\mathbb{K}}
\newcommand{\vol}{{\rm vol}}
\newcommand{\area}{{\rm area}}
\newcommand{\sys}{{\rm sys}}
\newcommand{\cf}{{\it cf.}}
\newcommand{\ie}{{\it i.e.}}

\long\def\forget#1\forgotten{} %

\title{Optimal systolic inequalities on Finsler Mobius bands}


\author[S. Sabourau and Z.Yassine]{St\'ephane Sabourau and Zeina Yassine}

\address{
Universit\'e Paris-Est, Laboratoire d'Analyse et Math\'ematiques Appliqu\'ees
(UMR 8050), UPEC, UPEMLV, CNRS, F-94010  \\ 
Cr\'eteil\\
France}
\email{stephane.sabourau@u-pec.fr}

\email{zeina.yassine@u-pec.fr}

\thanks{Partially supported by the French ANR project FINSLER } 

\begin{document}

\vspace{18mm} \setcounter{page}{1} \thispagestyle{empty}

\begin{abstract}
We prove optimal systolic inequalities on Finsler Mobius bands relating the systole and the height of the Mobius band to its Holmes-Thompson volume. 
We also establish an optimal systolic inequality for Finsler Klein bottles of revolution, which we conjecture to hold true for arbitrary Finsler metrics. 
Extremal metric families both on the Mobius band and the Klein bottle are also presented.
\end{abstract}

\maketitle

\section{Introduction} 
Optimal systolic inequalities were studied since the mid-twentieth century after C. Loewner proved in an unpublished work the following result, \cf~\cite{Ka07}. 
Every Riemannian two-torus $\T^2$ satisfies 
\begin{equation}\label{loewner}
\area(\T^2) \geq \frac{\sqrt{3}}{2} \sys^2(\T^2)
\end{equation}
with equality if and only if $\T^2$ is a flat hexagonal torus. 
Recall that the \emph{systole} of a nonsimply connected Riemannian surface~$M$, denoted by~$\sys(M)$, represents the length of the shortest noncontractible loop of~$M$. 
This inequality leads us to introduce the \emph{systolic area} of~$M$ defined as 
\begin{equation}\label{sysvol}
\sigma_{R}(M):= \inf_g \frac{\area(M,g)}{\sys^2(M,g)}.
\end{equation}
where $g$ runs over all the Riemannian metrics on $M$ (hence the subscript~$R$ for Riemannian). 
Thus, $\sigma_{R}(\T^2)=\frac{\sqrt{3}}{2}$.
Following this direction, P.~Pu~\cite{Pu52} showed that $\sigma_{R}(\RP^2)=\frac{2}{\pi}$, where the infimum is attained exactly by the Riemannian metrics with constant (positive) curvature on the projective plane $\RP^2$. 
In the eighties, C.~Bavard~\cite{Ba86} proved that $\sigma_{R}(\K^2)=\frac{2\sqrt{2}}{\pi}$, where the infimum on the Klein bottle~$\K^2$ is not attained by a smooth Riemannian metric.
See also \cite{Sak88}, \cite{Ba88} and \cite{Ba06} for other proofs and variations on this inequality. 
These are the only nonsimply connected closed surfaces with a known systolic area.
The existence of extremal metrics in higher dimension is wide open.

\medskip

The original proofs of the optimal Riemannian systolic inequalities on $\T^2$, $\RP^2$ and $\K^2$ rely on the conformal representation theorem (a consequence of the uniformization theorem on Riemann surfaces) and proceed as follows.
By the uniformization theorem, every Riemannian metric~$g$ on a closed surface is conformally equivalent to a Riemannian metric~$g_0$ of constant curvature.
Taking the average of~$g$ over the isometry group of~$g_0$ gives rise to a new metric~$\bar{g}$ with the same area as~$g$.
By the Cauchy-Schwarz inequality, the systole of~$\bar{g}$ is at most the systole of~$g$.
Thus, the new metric~$\bar{g}$ has a lower ratio $\area / \sys^2$ than the original metric~$g$.
Now, if the isometry group of~$g_0$ is transitive, which is the case for $\T^2$ and~$\RP^2$, the metric~$\bar{g}$ has constant curvature.
Hence the result for the projective plane.
Then, it is not difficult to find the extremal metric among flat torus.
The case of the Klein bottle requires an extra argument since the isometry group of~$g_0$ is not transitive, \cf~Section~\ref{sec:klein}. \\

\forget
Indeed, let $\left(M,g\right)$ be a Riemannian  surface. 
Then by the uniformization theorem, $g$ is conformally equivalent to a Riemannian metric $g_0$ of constant curvature on $M$, \ie, there exists a function $f:M\rightarrow \mathbb{R}$ such that $g=f^2 g_0$. We deduce by the Cauchy-Schwarz inequality that  
\begin{equation}\label{uniformizationtheorem}
\frac{\area(M,g)}{\sys^2(M,g)}\geq\frac{\area(M,\bar{f}^2g_0)}{\sys^2(M,\bar{f}^2g_0)},
\end{equation}
where $\bar{f}$ is the integration of $f$ over the isometry group of $(M,g_0)$. Now, if the latter is transitive (this is the case of $\T^2$ and $\RP^2$), we conclude that $\bar{f}$ is constant and we can easily then deduce the extremal Riemannian metrics on $M$.
\forgotten

In this article, we consider Finsler systolic inequalities. 
Loosely speaking, a Finsler metric~$F$ is defined as a Riemannian metric except that its restriction to a tangent plane is no longer a Euclidean norm but a Minkowski norm, \cf~Section~\ref{preliminaries}. 
From a dynamical point of view, the function~$F^2$ can be considered as a Lagrangian which induces a Lagrangian flow on the tangent bundle~$TM$ of~$M$. 
Thus, Finsler manifolds can be considered as degree~$2$ homogeneous Lagrangian systems.
The trajectories of the Lagrangian correspond to the geodesics of the Finsler metric.

\medskip

There exist several definitions of volume for Finsler manifolds which coincide in the Riemannian case. 
We will consider the Holmes-Thompson volume~$\vol_{HT}$, \cf~Section~\ref{preliminaries}. 
As previously, we can define the systolic area~$\sigma_F$, with the subscript~$F$ for Finsler, by taking the infimum in~\eqref{sysvol} over all Finsler metrics on $M$.

\medskip

Contrary to the Riemannian case, there is no uniformization theorem for Finsler surfaces. 
As a result, the classical Riemannian tools to prove optimal systolic inequalities on surfaces, which are based on the conformal length method described above, do not carry over to the Finsler case.
New methods are thus required to deal with Finsler metrics.

\medskip

The first optimal Finsler systolic inequality has been obtained by S.~Ivanov \cite{Iv02, Iv11} who extended Pu's systolic inequality to Finsler projective planes. 

\begin{theorem}[\cite{Iv02,Iv11}]\label{Iv02}
Let $\RP^2$ be a Finsler projective plane.
Then
$$
\frac{\vol_{HT}(\RP^2)}{\sys^2(\RP^2)}\geq\frac{2}{\pi}.
$$
Furthermore, equality holds if all the geodesics are closed of the same length.
\end{theorem}

In particular, the systolic area of the projective plane is the same in the Riemannian and Finsler settings, that is, 
$$
\sigma_R(\RP^2)=\sigma_F(\RP^2)=\frac{2}{\pi}.
$$
Note that Theorem \ref{Iv02} provides an alternate proof of Pu's inequality in the Riemannian case which does not rely on the uniformization theorem. 

\medskip 

Using a different method based on~\cite{Gr99} and~\cite{BI02}, a Finsler version of Loewner's inequality \eqref{loewner} has been obtained by the first author \cite{Sa10}.

\begin{theorem}[\cite{Sa10}]\label{Sa10}
Let $\T^2$ be a Finsler two-torus.
Then 
$$
\frac{\vol_{HT}(\T^2)}{\sys^2(\T^2)}\geq\frac{2}{\pi}.
$$
Equality holds if $\T^2$ is homothetic to the quotient of $\mathbb{R}^2$, endowed with a parallelogram norm $||.||$, by a lattice whose unit disk of $||.||$ is a fundamental domain.
\end{theorem}

Observe that $\sigma_F(\T^2)=\sigma_F(\RP^2)$ contrary to the Riemannian case. An optimal Finsler systolic inequality holds for non-reversible Finsler metrics on $\T^2$, \cf~\cite{ABT}. Note also that there is no systolic inequality for non-reversible Finsler two-tori if one considers the Busemann volume instead of the Holmes-Thompson volume, \cf~\cite{AB}. 


\medskip

No systolic inequality holds for manifolds with boundary either.
However, P.~Pu \cite{Pu52} and  C.~Blatter \cite{Bl62} obtained optimal Riemannian systolic inequalities in each conformal class of the Mobius band and described the extremal metrics, \cf~Section~\ref{sec:candidates}. 
Later, these  inequalities were used by C.~Bavard \cite{Ba86} and T.~Sakai \cite{Sak88} in their proofs of the systolic inequality on the Klein bottle in the Riemannian case. 
The proof of the optimal conformal Riemannian systolic inequalities on the Mobius band relies on the uniformization theorem and the conformal length method (as in the original proofs of the Riemannian systolic inequalities on $\T^2$, $\RP^2$ and~$\K^2$).

\forget
among manifolds with boundary, one cannot expect to get a positive systolic constant. We only know that in the Riemannian case, the Mobius band $\M$ does have an extremal metric in each conformal class, \cf~\cite{Ba88}. P.~Pu \cite{Pu52} and  C.~Blatter \cite{Bl62} obtained conformal systolic inequalities on $\M$, \cf~Section~\ref{sec:candidates}. 
These inequalities were used later by C.~Bavard \cite{Ba86} and T.~Sakai \cite{Sak88} in their proofs of the systolic inequality on the Klein bottle in the Riemannian case. 
As in the original proofs of the Riemannian systolic inequalities on $\T^2$, $\RP^2$ and~$\K^2$, the uniformization theorem was also essential in the proof of the conformal systolic inequalities on $\M$. 
\forgotten

\medskip

In this article, we first prove a Finsler generalization of the optimal systolic inequality on~$\T^2$ extending Loewner's inequality, \cf~\cite{Ke67}, and derive further optimal geometric inequalities on Finsler cylinders, \cf~Section~\ref{sec:torus}. 
These results allow us to establish an optimal inequality on every Finsler Mobius band~$\M$ relating its systole~$\sys(\M)$, its height~$h(\M)$ and its (Holmes-Thompson) volume~$\vol_{HT}(\M)$ at least when $\M$ is wide enough, \cf~Section~\ref{sec:wide}.
Here, the height~$h(\M)$ represents the minimal length of arcs with endpoints on the boundary~$\partial \M$, which are not homotopic to an arc in~$\partial \M$, \cf~Definition~\ref{heightdefinition}.
More precisely, we prove the following.

\begin{theorem}\label{finslermobius}
Let $\M$ be a Finsler Mobius band. Let $\lambda:=\frac{h(\M)}{\sys(\M)}$. 
Then
\begin{equation}\label{FM}
\frac{\vol_{HT}(\M)}{\sys(\M) \, h(\M)} \geq
\begin{cases}
\frac{2}{\pi} &  \text{if } \lambda \in (0,1] \\
\frac{1}{\pi}\frac{\lambda+1}{\lambda} & \text{otherwise}.
\end{cases}
\end{equation}
Moreover, the above inequalities are optimal for every value of $\lambda\in(0,+\infty)$.
\end{theorem}

We describe extremal and almost extremal metric families in details in Section~\ref{sec:candidates}, Example~\ref{ex:extremal-wide} and Example~\ref{ex:extremal-narrow}. \\

The optimal Finsler systolic inequality on the Klein bottle is still unknown. 
However, based on the inequality~\eqref{FM} on Finsler Mobius bands, we obtain a partial result for Finsler Klein bottles with nontrivial symmetries. 
We refer to Definition~\ref{def:sym} for a description of the symmetries considered in the statement of the following theorem.

\begin{theorem}\label{finslerkleinbottle}
Let $\K$ be a Finsler Klein bottle with a soul, soul-switching or rotational symmetry. Then
\begin{equation} \label{k}
\frac{\vol_{HT}(\K)}{\sys^2(\K)}\geq \frac{2}{\pi}.
\end{equation}
Moreover, the inequality is optimal.
\end{theorem}

We also present some extremal metric family in Example~\ref{ex:extremal2}. \\

Finally, we present as a conjecture that the inequality~\eqref{k} should hold for every Finsler Klein bottle with or without symmetries.
That is, $\sigma_F(\K)$ should be equal to~$\frac{2}{\pi}$ (as $\sigma_F(\T^2)$ and $\sigma_F(\RP^2)$).

\section{Preliminaries}\label{preliminaries}

In this section, we introduce general definitions regarding Finsler manifolds.

\medskip

A (reversible) \emph{Finsler metric} $F : TM \rightarrow [0,+\infty)$ on the tangent bundle $TM$ of a smooth n-dimensional manifold $M$ is a continuous function satisfying the following conditions (for simplicity, let $F_{x}:=F|_{T_xM}$):
\begin{enumerate}
\item Smoothness: $F$ is smooth outside the zero section; \label{1}
\item Homogeneity: $F_{x}(tv) = |t|F_{x}(v)$ for every $v \in T_xM$ and $t \in \mathbb{R}$; \label{2}
\item Quadratic convexity: for every $x \in M$, the function $F_{x}^2$ has positive definite second derivatives on $T_xM\setminus{0}$, \ie, if $p$, $u$, $v\in T_{x}M$, the symmetric bilinear form
\begin{equation*}
g_p(u,v):=\frac{1}{2}\frac{\partial^2}{\partial s\partial t}\left(F^{2}_{x}(p+tu+sv)\right)|_{t=s=0}
\end{equation*}
is an inner product. \label{3}
\end{enumerate}

The pair $(M, F)$ is called a \emph{Finsler manifold}. If $F$ is only positive homogeneous instead of homogeneous, that is, \eqref{2} only holds for $t\geq 0$, we say that the Finsler metric is non-reversible. For simplicity, we will only consider reversible Finsler metrics.
 
Conditions \eqref{1}, \eqref{2} and \eqref{3} imply that $F$ is strictly positive outside the zero section and that for every $x \in M$ and $u,v \in T_xM$, we have 
\begin{equation*}
F_x(u+v)\leq F_x(u)+F_x(v),	
\end{equation*}
with equality if and only if $u=\lambda v$ or $v=\lambda u$ for some $\lambda\geq 0$, \cf~\cite[Theorem~1.2.2]{BCS00}. 
Hence, $F$ induces a strictly convex norm $F_{x}$ on each tangent space $T_{x}M$ with $x\in M$. 
More specifically, it gives rise to a Minkowski norm~$F_x$ on each tangent symmetric space~$T_x M$.
Working with quadratically convex norms and not merely (strictly) convex norms provides nice dynamical features such as a geodesic flow and a Legendre transform, \cf~\cite[\S1]{Be78}.

\medskip

As in the Riemannian case, notions of length, distance, and geodesics extend to Finsler geometry. Let $\gamma:[a,b]\rightarrow M$ be a piecewise smooth curve. The length of $\gamma$ is defined as 
\[
\ell(\gamma):=\int^{b}_{a}F(\dot{\gamma}(t))dt.
\]
By condition~\eqref{2}, $\ell(\gamma)$ does not depend on the parametrization of~$\gamma$. 
Moreover, the functional $\ell$ gives rise to a distance function $d_F : M \times M \rightarrow [0,\infty)$ defined as $d_{F}(x, y)=\inf_\gamma \ell(\gamma)$, where the infimum is taken over all piecewise smooth curves $\gamma$ joining $x$ to $y$. 
A geodesic is a curve which locally minimizes length. It is a critical point of the energy functional $\gamma\mapsto \int F^2(\gamma(t)) dt$ (here the quadratic convexity condition~\eqref{3} is necessary).

\medskip

For $x \in M$, we denote by $B_x$ the unit ball of the Minkowski norm $F_x$, that is, $B_x = \left\{v \in T_xM\cong\mathbb{R}^n \mid F_x(v) \leq 1\right\}$. Furthermore, the norm $F_x$ is Euclidean if and only if $B_x$ is an ellipsoid. 
The dual of $B_x$ is defined as $B^{*}_{x}=\left\{f\in T_{x}^{*}M\mid F_{x}^{*}(f)\leq 1\right\}$ where $F^{*}_{x}$ is the dual norm of $F_{x}$. 
Note that $B^{*}_{x}$ identifies with the polar body $B^{\circ}_{x}=\left\{ u \in T_xM\mid \langle u,v\rangle\leq 1 \text{ for every } v\in B_x\right\}$ of $B_x$, where $\langle.,.\rangle$ is a given scalar product on $T_xM$.

\medskip

In the Riemannian case, there exists a unique notion of volume, up to normalization, which agrees both with the n-dimensional Hausdorff measure determined by the Riemannian metric and with the projection of the Liouville measure from the unit tangent bundle, \cf~\cite[\S5.5]{BBI01}. 
However, in the Finsler case, there is no notion of volume that satisfies both properties, \cf~\cite{BI12}. This leads to two distinct notions of Finsler volume presented below.

Denote by $\varepsilon_{n}$ the Euclidean volume of the Euclidean unit ball in $\mathbb{R}^n$. 
Let $dx$ represent a given volume form on~$M$ and $m$ be the restriction of this volume form to each tangent space $T_{x}M$. Similarly, let $m^*$ be the restriction of the volume form dual to~$dx$ to each cotangent space $T^*_xM$. The \emph{Busemann volume}, \cf~\cite{Bu47}, is defined as
\begin{equation}\label{busemann}
\vol_{B}(M):=\int_{M}\frac{\varepsilon_{n}}{m(B_x)}dx.
\end{equation}
The Busemann volume is sometimes called the Busemann-Hausdorff volume as it agrees with the $n$-dimensional Hausdorff measure of $M$ (at least when the Finsler metric $F$ is reversible).
Another volume frequently used in Finsler geometry is the \emph{Holmes-Thompson volume}, \cf~\cite{HT79}, defined as 
\begin{equation}\label{holmesthompson}
\vol_{HT}(M):=\int_{M}\frac{m^*(B^{*}_{x})}{\varepsilon_{n}} \, dx.
\end{equation}
It is equal to the Liouville volume of its unit co-disc bundle divided by $\varepsilon_n$, \cf~\cite{AT04}. Note that the integrals in \eqref{busemann} and \eqref{holmesthompson} do not depend on the chosen volume form.
Sine the volume is a local notion, it is possible to extend this definition even when $M$ is nonorientable, that is, when volume forms do not exist on~$M$.

In \cite{Du98}, C. Dur\'an proved the following volume comparison inequality for Finsler manifolds: 
\begin{equation}\label{Du98}
\vol_{HT}(M)\leq \vol_{B}(M)
\end{equation}
with equality if and only if $M$ is Riemannian. Hence, every systolic inequality for the Holmes-Thompson volume remains true for the Busemann volume. However, the inequality \eqref{Du98} may fail for non-reversible Finsler metrics.

\section{A systolic inequality on Finsler two-tori} \label{sec:torus}
In this section we establish a Finsler version of the Minkowski second theorem for the two-torus. 
More precisely, L. Keen proved the following.

\begin{proposition}[\cite{Ke67}, \cite{Ka07} \S6.2] \label{Ke67}
Let $\T^2$ be a Riemannian two-torus. There exist two closed curves of lengths $a$ and $b$ generating the first integral homology group of $\T^2$ such that
\begin{equation*}
\vol_{HT}(\T^2)\geq\frac{\sqrt{3}}{2} ab.
\end{equation*}
Equality holds if and only if $\T^2$ is homothetic to the flat torus obtained as the quotient of $\mathbb{R}^2$ by a hexagonal lattice.   
\end{proposition}

The proof of Proposition~\ref{Ke67} relies on the uniformization theorem and the Cauchy-Schwarz inequality.


\medskip

A Finsler version of Proposition \ref{Ke67} is given by the following result.

\begin{proposition}\label{Ke67finsler}
Let $\T^2$ be a Finsler two-torus. There exist two closed curves of lengths $a$ and $b$ generating the first integral homology group of $\T^2$ such that 
\begin{equation*}
\vol_{HT}(\T^2)\geq\frac{2}{\pi} ab.
\end{equation*}
Equality holds if $\T^2$ is homothetic to the quotient of $\mathbb{R}^2$, endowed with a parallelogram norm $||.||$, by a lattice generated by two vectors of lengths $a$ and $b$, parallel to the sides of the unit ball of $||.||$.   
\end{proposition}

Since there is no uniformization theorem for Finsler metrics, the proof of this proposition differs from the proof of Proposition \ref{Ke67}.

\begin{proof}
Let $\alpha$ be a systolic loop of~$\T^2$ and $\beta$ be the shortest closed curve of~$\T^2$ homologically independent with~$\alpha$.
Denote by~$a$ and~$b$ the lengths of~$\alpha$ and~$\beta$.
The loops~$\alpha$ and~$\beta$ are simple and intersect each other at a single point.
Cutting $\T^2$ open along~$\alpha$ and~$\beta$ gives rise to a surface~$\Delta$ isometric to a fundamental domain of~$\T^2$.
Let $L$ be a positive number greater than $max\{a,b\}$. 
Denote by $p$ and $q$ the smallest integers such that $pa\geq L$ and $qb\geq L$.  
Then, glue $pq$ copies of $\Delta$ in such a way that the resulting shape is isometric to the fundamental domain of a Finsler torus of volume $pq$ times the volume of $\T^2$ and of systole equal to $min\{pa,qb\}$. 
By construction, this new Finsler torus is a degree~$pq$ cover of $\T^2$. 
Then, by Theorem~\ref{Sa10}, we have
$$
pq \, \vol_{HT}(\T^2)\geq \frac{2}{\pi}\left(\min\{pa,qb\}\right)^{2}.
$$
Hence,
$$
\vol_{HT}(\T^2)\geq \frac{2}{\pi}\frac{L}{p}\frac{L}{q}\geq \frac{2}{\pi}\frac{p-1}{p}\frac{q-1}{q}ab.
$$
By choosing $L$ large enough, the integers $p$ and $q$ can be made arbitrarily large, which leads to the desired inequality.

Now, if $\T^2$ is the quotient of $\mathbb{R}^2$, endowed with a parallelogram norm, by a lattice generated by two vectors of lengths $a$ and $b$ which are parallel to the sides of the unit ball of the parallelogram norm, then $\vol_{HT}(\T^2)=\frac{2}{\pi}ab$.
\end{proof}

\begin{remark}
Briefly speaking, the idea of the proof of Proposition \ref{Ke67finsler} is to use finite covers to get a quasi-isosystolic two-torus (\ie, whose first homology group is generated by two loops of lengths nearly the systole) and to apply the systolic inequality of Theorem~\ref{Sa10} to this two-torus. 
This argument also applies in the Riemannian case and gives an alternative proof of Proposition \ref{Ke67} without the use of the uniformization theorem. 
\end{remark}

We can apply Proposition \ref{Ke67finsler} to prove a systolic inequality on Finsler cylinders. First, we give the following definition

\begin{definition}\label{heightdefinition}
Let $M$ be a compact Finsler surface with boundary. The \emph{height} $h(M)$ of $M$ is the minimal length of arcs with endpoints on the boundary $\partial M$, which are not homotopic to an arc in~$\partial M$. More formally, 
$$
h(M):=\inf\{\ell(\gamma)\vert \gamma:[0,1]\rightarrow M \mbox{ with } \gamma(0), \gamma(1) \in \partial M \text{ and } [\gamma]\neq 0\in \pi_1(M,\partial M)\}.
$$
A \emph{height arc} of~$M$ is a length-minimizing arc of~$\gamma$ of~$M$ with endpoints in~$\partial M$ inducing a nontrivial class in $\pi_1(M,\partial M)$.
By definition, the length of a height arc of~$M$ is equal to~$h(M)$.
\end{definition}

\begin{proposition}\label{finslercylindre}
Let $C$ be a Finsler cylinder. Then, 
$$
\vol(C)\geq\frac{2}{\pi} \, \sys(C)h(C).
$$
\end{proposition}

\begin{proof}
Let $k$ be a positive even integer. We glue $k$ copies of $C$ by identifiying the identical boundary components pairwise. 
The resulting space is a torus~$\T^2$.
Every loop of~$\T^2$ non freely homotopic to a multiple of a boundary component of~$C$ is of length at least $k \, h(C)$.
Therefore, for symmetry reasons, if $k$ satisfies $k \, h(C) \geq \sys(C)$, the systole of the torus~$\T^2$ is equal to the systole of the cylinder $C$. 
Applying Proposition~\ref{Ke67finsler} to this torus, we derive
\begin{equation*}
\vol(\T^2)=k \, \vol(C)\geq\frac{2}{\pi} \, k \, \sys(C)h(C).
\end{equation*}
Hence the result.
\end{proof}

We will make use of Proposition~\ref{finslercylindre} in the proof of Theorem~\ref{finslermobius} for wide Finsler Mobius bands, \cf~Section~\ref{sec:wide}.

\section{Natural candidates for extremal metrics} \label{sec:candidates}
In this section, we first review  the extremal Riemannian metrics for systolic inequalities on the Mobius band and the Klein bottle presented in \cite{Pu52, Bl62, Ba86, Ba88, Sak88}. By analogy with the Riemannian metrics, we construct Finsler metrics which are natural to consider when studying optimal Finsler systolic inequalities.

\medskip
Consider the standard sphere $S^2$. Denote by $u$ and $v$ the longitude and the latitude on $S^2$. 
Let $a\in (0,\frac{\pi}{2})$. 
The $a$-tubular neighborhood of the equator $\{v=0\}$ is a spherical band $S_a$ which can be represented as 
$$
S_a:=\{(u,v)\mid-\pi\leq u\leq \pi,-a\leq v\leq a\}.
$$
The quotient of $S_a$ by the antipodal map is a Riemannian Mobius band with curvature~$1$ denoted by 
$\M_a$. 
The conformal modulus space of the Mobius band is parametrized by $\M_a$ with $a\in (0,\frac{\pi}{2})$. 
More precisely, every conformal class on the Mobius band agrees with the conformal structure induced by some $\M_a$ with $a\in(0,\frac{\pi}{2})$. 
Furthermore, the conformal classes of the $\M_a$'s are pairwise distinct.

\medskip
The spherical Mobius bands $\M_a$ are involved in several extremal conformal systolic inequalities for Riemannian metrics. More precisely, we define the \emph{orientable systole} of a Riemannian Mobius band $\M$ as the shortest length of a noncontractible orientable loop in $\M$. It will be denoted by $\sys_{+}(\M)$. Similarly, we define the \emph{nonorientable systole} of $\M$ and denote it by $\sys_{-}(\M)$. 
Observe that $\sys(\M)=\min\left\{\sys_{+}(\M),\sys_{-}(\M)\right\}$. 
Moreover, we define $\ell_v(\M)$ as the minimal length of the arcs joining $\left(u,-a\right)$ to $\left(u,a\right)$ in the $\left(u,v\right)$-coordinates of $\M_a$, which are homotopic with fixed endpoints to the projections of the meridians in~$S_a$.
For instance, $\sys_{+}(\M_a)=2\pi\cos(a)$, $\sys_{-}(\M_a)=\pi$ and $\ell_v(\M_a)=2a$. 
Note that the definition of~$\ell_v$ relies on conformal data, namely the longitude-latitude coordinates to define the endpoints of the arcs involved in the length minimization.

\medskip
In \cite{Bl62}, C. Blatter obtained optimal lower bounds for the functionals 
$$
\sigma_1:=\frac{\vol}{\sys_{-}^2},\quad \sigma_2:=\frac{\vol}{\sys_{-} \times \ell_v} \quad \text{and} \quad  \sigma_3:=\frac{\vol}{\sys \times \ell_v}
$$ 
in each conformal class of the Mobius band. More precisely, for every Riemannian metric conformally equivalent to $\M_a$, we have the sharp lower bound
\begin{equation}\label{sigma1}
\sigma_1(\M)\geq \sigma_1(\M_a).
\end{equation}
We also have the sharp inequality 
\begin{equation}\label{sigma2}
\sigma_2(\M)\geq
\begin{cases}
\sigma_2(\M_a)&  \text{if } a \in (0,b] \\
\sigma_2(\M_{\alpha(a)}\cup C_{a,\alpha(a)})&  \text{if } a \in [b,\frac{\pi}{2})
\end{cases}
\end{equation}
where $b$ is the unique solution in $(0,\frac{\pi}{2})$ of the equation $\tan(x)=2x$ and $\M_{\alpha(a)}\cup C_{a, \alpha(a)}$ is the Mobius band obtained by attaching a flat cylinder $C_{a,\alpha(a)}$ to the spherical Mobius band $\M_{a,\alpha(a)}$ along their boundary.
Here, the angle $\alpha(a)\in[b, a]$ is implicitly given by a nonlinear equation depending on the conformal type $a$ and the flat cylinder $C_{a, \alpha(a)}$ is defined as the product $\partial_+ S_\alpha(a) \times [0,\sin a - \sin \alpha(a) ]$, where $\partial_+ S_{\alpha(a)}$ is a boundary component of~$S_\alpha(a)$.
Alternately, $C_{a, \alpha(a)}$ is the Mercator projection of a connected component of $S_a\setminus S_{\alpha(a)}$ to the vertical cylinder generated by~$\partial S_{\alpha(a)}$. \\
Finally, we have the third sharp inequality 
\begin{equation}\label{sigma3}
\sigma_3(\M)\geq
\begin{cases}
\sigma_3(\M_a)&  \text{if } a\in(0,\frac{\pi}{3}]\\
\sigma_3(\M_{\frac{\pi}{3}}\cup C_{a,\frac{\pi}{3}})& \text{if } a\in[\frac{\pi}{3},\frac{\pi}{2})
\end{cases}
\end{equation}

With the help of~\eqref{sigma1}, C.~Bavard~\cite{Ba86} established the optimal isosystolic inequality on the Klein bottle. 
Later, T.~Sakai~\cite{Sak88} used the inequalities~\eqref{sigma2} and~\eqref{sigma3} to give an alternative proof of Bavard's isosystolic inequality for the Klein bottle. 
The extremal Riemannian metric on the Klein bottle is obtained by gluing two copies of the spherical Mobius band~$\M_{\frac{\pi}{4}}$ along their boundary.

\medskip
The closed geodesics in $S_a$ project down to systolic loops in $\M_a$ which differ only by rotations and can be described as follows. Let $\gamma^0_0$ be the equator $\{v=0\}$ of $S^{2}$ parametrized by arclength. 
Every great circle $\gamma$ in~$S_a$ different from~$\gamma^0_0$ intersects~$\gamma^0_0$ at a unique point~$\gamma^0_0(s)$ with $s\in [0,\pi)$. 
The great circles $\gamma^0_0$ and $\gamma$ form at $\gamma^0_0(s)$ an oriented angle $\vartheta\in [-\frac{\pi}{2},\frac{\pi}{2}]$ with $\vartheta \neq 0$. 
Such great circle $\gamma$ is denoted by $\gamma^s_\vartheta$. 
From the sinus formula in spherical trigonometry, the great circle $\gamma^s_\vartheta$ exactly lies between the circles of latitude $\pm \vartheta$. 
Thus, $\vartheta\in [-a,a]$. 

For $\vartheta>0$, let $v \in(0,\vartheta]$. 
By Clairaut's relation, \cf~\cite[\S4.4]{DoC76}, the positive angle $\theta_{\vartheta}$ between the circle of latitude~$v$ and the great circle $\gamma^s_\vartheta$ satisfies $\theta_{\vartheta}(v)=\arccos\left(\frac{\cos(\vartheta)}{\cos(v)}\right)$. 
In particular, the unit tangent vectors to the systolic loops of $\M_a$ at a point of latitude~$v$ generate a symmetric cone of half angle 
\[
\theta(v)=\arccos\left(\frac{\cos(a)}{\cos(v)}\right)
\] 
in the tangent plane, \cf~Figure~\ref{extremalriemannian}. 
These unit tangent vectors of~$\M_a$ are referred to as \emph{systolic directions}. 
The unit tangent vectors to the meridians are called \emph{meridian directions}. 
Despite the risk of confusion, we will also call great circles of~$\M_a$ the projections of the great circles of~$S_a$ to~$\M_a$.


\begin{figure}[htbp]
\includegraphics[width=3cm]{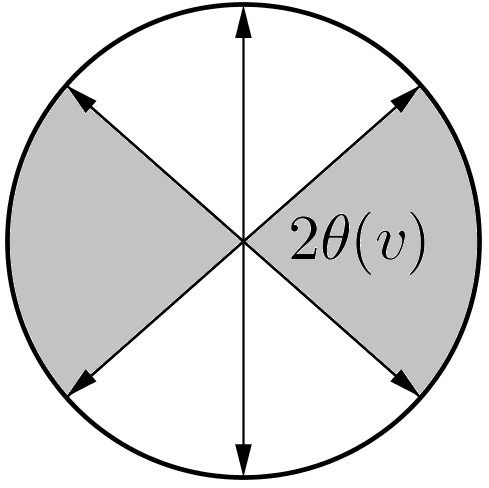}
\caption[extremalriemannian]{Systolic (in gray) and meridian directions in the unit tangent plane at a point of latitude~$v$ in $\M_a$.} \label{extremalriemannian}
\end{figure}

\medskip
The spherical Mobius band $\M_a$, which is extremal for some conformal systolic inequalities among Riemannian metrics, \cf~\eqref{sigma1}, \eqref{sigma2}, and \eqref{sigma3}, is not extremal among Finsler metrics. 
Indeed, by slightly perturbing the quadratically convex norm in each tangent plane away from the systolic and meridian directions of the spherical metric, \cf~Figure~\ref{extremalriemannian}, we can decrease the area of the Mobius band without changing the systole and the height. 
This shows that any unit tangent vector to an extremal (quadratically convex) Finsler Mobius band is tangent either to a systolic loop or a height arc, \cf~Definition~\ref{heightdefinition}.
In other words, the unit tangent vectors induced by the systolic loops and the height arcs of an extremal (quadratically convex) Finsler Mobius band fill in its unit tangent bundle.

\medskip

With this observation in mind, it is natural to consider the following (non-quadratically convex) Finsler metrics as potential extremal metrics. 
The idea is to adjust the shapes of the unit balls in the tangent bundle of the Mobius band so that the systolic and meridian directions  fill in the unit tangent bundle. More precisely, define a Finsler metric $F_a$ on $S_a$ whose restriction to each tangent plane $T_{x}S_a$ is a norm $F_a|_{x}$ of unit ball $B_{x}$ given by the convex hull of the systolic directions of $\M_a$, \cf~Figure~\ref{sabourauextremal}. 
In longitude and latitude coordinates, the ball $B_{x}$ at $x=(u,v)$ can be represented as
$$
B_{x}:=\{(\xi_u,\xi_v)\in T_{x}S_a\mid \xi_u^2+\xi_v^2\leq 1, |\xi_v|\leq \sin \theta(v) \}.
$$
Hence, the Finsler metric $F_a$ can be represented in local coordinates as
\begin{equation*}
F_a =
\begin{cases}
 \frac{1}{\sin(\theta(v))} \, |dv| &  \text{if } \arctan\left(\frac{dv}{du}\right) \in [0,\theta(v)] \\
\sqrt{du^2+dv^2}&  \text{if } \arctan\left(\frac{dv}{du}\right) \in [\theta(v),\frac{\pi}{2}]
\end{cases}
\end{equation*}
This metric passes to the quotient by the antipodal map to a Finsler metric still denoted by $F_a$. Denote by $\M_{F_a}$ the Finsler Mobius band so obtained. 
\begin{figure}[htbp]
\includegraphics[width=3cm]{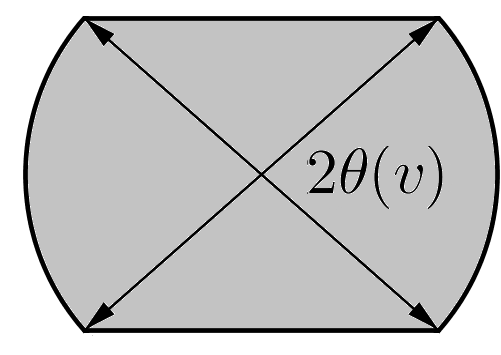}
\caption[sabourauextremal]{Unit ball of $F_{a}$ at a point of latitude $v$ in $S_a$.} \label{sabourauextremal}
\end{figure}

Since the definition of~$\ell_v$ relies on conformal data which do not extend to the Finsler case, it is more natural to consider the height $h(\M)$, \cf~Definition~\ref{heightdefinition}, in the Finsler case. 

\medskip


Some geometric features of the Finsler Mobius bands $\M_{F_a}$ are summarized in the following two propositions.  

\begin{proposition}\label{height}
Let $a\in(0,\frac{\pi}{2})$. Then, $\sys_{-}(\M_{F_a})=\pi$, $\sys_{+}(\M_{F_a})=2\pi\cos(a)$ and $h(\M_{F_a})=\pi$. 

In particular, if $a=\frac{\pi}{3}$ then $\sys(\M_{F_a})=\sys_{+}(\M_{F_a})=\sys_{-}(\M_{F_a})=h(\M_{F_a})=\pi$.  
\end{proposition}

\begin{proof}
Let us start with a useful observation.
Denote by~$\mathcal{S}$ the interior of the domain of~$U\M_{F_a}$ formed by the unit tangent vectors of the great circles of~$\M_a$.
The Finsler metric~$F_a$ coincides with the round Riemannian metric of~$\M_a$ on~$\mathcal{S}$.
Therefore, the subset~$\mathcal{S}$ is stable under the geodesic flow of~$F_a$ (which is well-defined on~$\mathcal{S}$).
Furthermore, the length of a great circle with respect to~$F_a$ is equal to~$\pi$.

\medskip

Let us show that $h(\M_{F_a})=\pi$.
Consider a height arc~$\gamma$ of~$\M_{F_a}$.
%
The arc~$\gamma$ can be parametrized with respect to the latitude.
Otherwise, we could remove a subarc of~$\gamma$ joining two points at the same latitude and still make up an arc in the same relative homotopy class as~$\gamma$ with the remaining pieces using the rotational symmetry of~$\M_{F_a}$.
This would contradict the length-minimizing property of~$\gamma$.
Hence,
\begin{align*}
h(\M_{F_a})
&=\ell(\gamma)\\
&=\int_{-a}^a \frac{1}{\sin \theta(v)} \, dv\\
&=\int_{-a}^a\frac{\cos(v)}{\sqrt{\cos^{2}(v)-\cos^{2}(a)}} \, dv \\
&=2\arctan\left(\frac{\sqrt{2}\sin(v)}{\sqrt{\cos(2v)-\cos(2a)}}\right)\bigg|_{0}^a\\
&=\pi.
\end{align*}  

Now, let us show that the systolic curves of~$\M_{F_a}$ agree with the great circles of~$\M_a$ in the nonorientable case and with the boundary of~$\M_a$ in the orientable case.
Consider an orientable or nonorientable noncontractible loop~$\gamma$ of minimal length in~$\M_{F_a}$.

If $\gamma$ lies in the boundary of~$\M_{F_a}$ then the loop~$\gamma$ is orientable of length $2\pi \cos(a)$.
Thus, we can assume that $\gamma$ passes through an interior point~$p$ of~$\M_{F_a}$.

If a tangent vector of~$\gamma$ lies in~$\mathcal{S}$ then the geodesic arc~$\gamma$ coincides with a great circle of~$\M_a$ in the nonorientable case and with a great circle run twice in the orientable case. 
(Recall that $\mathcal{S}$ is stable by the geodesic flow of~$F_a$.)
In the former case, the curve~$\gamma$ is of length~$\pi$, while in the latter, it is of length~$2\pi$.
Thus, we can assume that the tangent vectors of~$\gamma$ do not lie in~$\mathcal{S}$.

Consider the closed lift~$\bar{\gamma}$ of~$\gamma$ in~$S_a$.
Let~$c_\pm$ be the two extreme great circles of~$S_a$ passing through the lifts of~$p$ and tangent to the boundary of~$S_a$.
That is, $c_\pm$ are the great circles of~$S_a$ making an angle of~$\pm \theta(v)$ with the curves of constant latitude~$\pm v$ in~$S_a$ passing through the lifts of~$p$.
Since the tangent vectors of~$\gamma$ do not lie in~$\mathcal{S}$, the curve~$\bar{\gamma}$ does not intersect~$c_\pm$ in the interior of~$S_a$, except at the lifts of~$p$.
Therefore, there exists a subarc of~$\bar{\gamma}$ (actually two subarcs of~$\bar{\gamma}$) joining the two boundary components of~$S_a$ in the region delimited by the great circles~$c_\pm$ and the boundary of~$S_a$,  see the gray region of Figure~\ref{propositionheight}.
Thus, $\ell(\gamma) \geq h(\M_{F_a}) = \pi$ with equality if $\gamma$ agrees with~$c_\pm$.

We conclude that $\sys_{-}(\M_{F_a})=\pi$ and $\sys_{+}(\M_{F_a})=\ell(\partial\M_{F_a})=2\pi\cos(a)$. Hence,
\begin{equation*}
\sys(\M_{F_a}) =
\begin{cases}
 \sys_{-}(\M_{F_a})&  \mbox{if } a \in (0,\frac{\pi}{3}] \\
 \sys_{+}(\M_{F_a})& \mbox{if } a \in [\frac{\pi}{3},\frac{\pi}{2})
\end{cases}
\end{equation*}
\end{proof}

\forget

\begin{proof}
We claim that the systolic curves are only the great circles. Observe first that the set of unit vectors in $U\M_{F_a}$ that are pointing in the directions of great circles is stable under the geodesic flow, \cf~\cite{Be78}, Chapter 1.F. Moreover, it is easy to see that the length of a great circle is equal to $\pi$. Now, let $\gamma$ be a geodesic loop in $\M_{F_a}$ that lifts to an arc joinning the point $(-\frac{\pi}{2},-b)$ to the point $(\frac{\pi}{2},b)$ in $S_a$, where $b\in[0,a]$, pointing in a meridian direction. Let $\gamma_{0}$ be a geodesic loop in $\M_{F_a}$ that lifts to the extreme great circle $\gamma_{a}^{s}$ in $S_a$ (that is, the angle between the meridian $\left\{v=\alpha\right\}$ and $\gamma_{a}^{s}$ is equal to $\arccos\left(\frac{\cos(a)}{\cos(\alpha)}\right)$, $\alpha\in[-b,b]$). By a calibration argument we have $\ell(\gamma)\geq\ell(\gamma_{0})=\pi$, with equality if and only if $\gamma$ coincides with $\gamma_{a}^{s}$. In fact, assume that the geodesics loops $\gamma$ and $\gamma_{0}$ are parametrized by the latitude $v$, then
\begin{align*}
\int^{b}_{-b}du|_{\gamma(v)}(\dot{\gamma}(v))\mathrm{d}v&\geq\int^{b}_{-b}du|_{\gamma_{0}(v)}(\dot{\gamma_{0}}(v))\mathrm{d}v.
\end{align*}
Hence, $\sys_{-}(\M_{F_a})=\pi$. Moreover, it is not hard to show that $\sys_{+}(\M_{F_a})=\ell(\partial\M_{F_a})=2\pi\cos(a)$. Hence,
\begin{equation*}
\sys(\M_{F_a}) =
\begin{cases}
 \sys_{-}(\M_{F_a})&  \text{if } a \in ]0,\frac{\pi}{3}] \\
 \sys_{+}(\M_{F_a})&  \text{if } a \in [\frac{\pi}{3},\frac{\pi}{2}[.
\end{cases}
\end{equation*}

Finally, let $\gamma$ be a unit speed (non closed) geodesic arc pointing in a meridian direction, and of length $h(\M_a)$. Then $\gamma$ joins two points of $\partial\M_{F_a}$ and induces a nontrivial class in $\pi_1(\M_{F_a},\partial\M_{F_a})$ s.  We claim that its length is in fact equal to $\pi$. Indeed, we have
\begin{align*}
h(\M_{F_a})
&=\ell(\gamma)\\
&=2\int_{0}^a \frac{1}{\sin(\theta)}\mathrm{d}v\\
&=2\int_{0}^a\frac{\cos(v)}{\sqrt{\cos^{2}(v)-\cos^{2}(a)}}\mathrm{d}v\\
&=2\arctan\left(\frac{\sqrt{2}\sin(v)}{\sqrt{\cos(2v)-\cos(2a)}}\right)\bigg|_{0}^a\\
&=\pi.
\end{align*}  
\end{proof}

\forgotten
\begin{figure}[htbp]
\includegraphics[width=5cm]{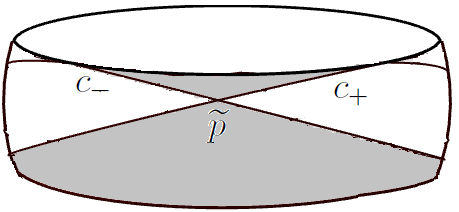}
\caption[uu]{The extreme great circles $c_\pm$ passing through a lift $\widetilde{p}$ of~$p$ in~$S_a$.} \label{propositionheight}
\end{figure}

\begin{remark}
The Finsler Mobius bands $\M_{F_a}$ are not pairwise isometric since they have distinct orientable systoles. 
\end{remark}

\begin{proposition} \label{volume}
Let $a\in(0,\frac{\pi}{2})$. Then, $\vol_{HT}(\M_{F_a})=2\pi.$ 
\end{proposition}

\begin{proof}
The unit ball $B_x^{*}$ coincides with the polar body of $B_x$ described in Figure~\ref{dualsabourauextremal}. 
The area of $B_x^{*}$ is equal to $2\theta(v)+\frac{2}{\tan \theta(v)}$. 
By definition of the Holmes-Thompson volume, \cf~\eqref{holmesthompson}, we have
\begin{align*}
\vol_{HT}(\M_{F_a})
&=\frac{1}{2\pi}\int_{S_{a}} m(B^{*}_{x})\, dm(x)\\
&=\frac{2}{\pi}\int^{\frac{\pi}{2}}_{-\frac{\pi}{2}}\int^{a}_{-a}\left( \theta(v)+\frac{1}{\tan\theta(v)}\right)\cos(v) \, du \, dv \\
&=\frac{2}{\pi}\int^{\frac{\pi}{2}}_{-\frac{\pi}{2}}\int^{a}_{-a}\left(\arccos\left(\frac{\cos(a)}{\cos(v)}\right)+\frac{\cos(a)}{\sqrt{\cos^2(v)-\cos^2(a)}}\right)\cos(v) \, du \, dv \\
&=2\pi.
\end{align*} 
\end{proof}

\begin{figure}[htbp]
\includegraphics[width=3cm]{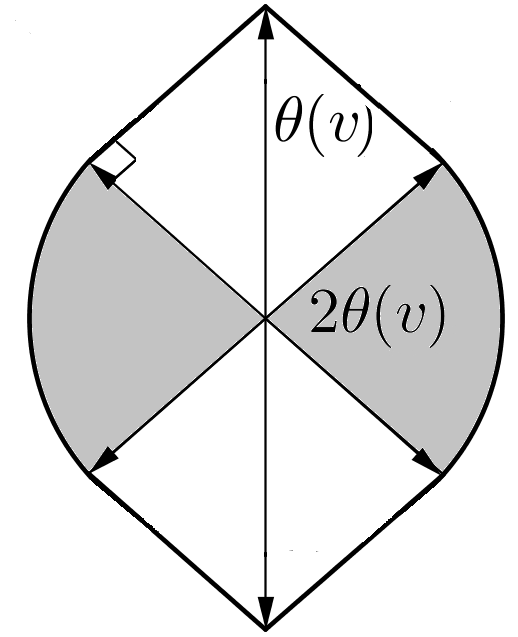}
\caption[dualsabourauextremal]{Dual unit ball $B_x^{*}$ or polar body of~$B_x$.} \label{dualsabourauextremal}
\end{figure}

\begin{remark}
As a consequence of Propositions \ref{height} and \ref{volume}, we observe the following couple of points.
\begin{enumerate}
\item The orientable and nonorientable systoles of~$F_a$ agree with those of its dual $F_a^{*}$. Hence, $\sys(\M_{F_a})=\sys(\M_{F_a^*})$. 
Moreover, computations similar to those in Propositions~\ref{height} and~\ref{volume} show that  $h(\M_{F_a^*})=\pi(1-\cos(a))$ and $\vol_{HT}(\M_{F_{a}^*})=2\pi\sin^2(a)$. 
This means that for both $F_a$ and its dual $F_a^{*}$, we have $\frac{\vol_{HT}(\M)}{\sys_{-}(\M) \, h(\M)}\rightarrow\frac{2}{\pi}$, when $a\rightarrow\frac{\pi}{2}$. 

\medskip

\item The Finsler Mobius bands $\M_{F_{a}}$ with $a\in (0,\frac{\pi}{3}]$ attain the equality case in~\eqref{FM} when $\sys(\M)=h(\M)$.
\end{enumerate}
\end{remark}

\section{Systolic inequalities on wide Finsler Mobius bands} \label{sec:wide}
In this section, we give a proof of Theorem~\ref{finslermobius} for wide Finsler Mobius bands, that is, when $\lambda \geq 1$.
More precisely, we prove the following result.

\begin{proposition} \label{prop:wide}
Let $\M$ be a Finsler Mobius band with $h(\M) \geq \sys(\M)$.
Then
\begin{equation} \label{eq:wide}
\vol_{HT}(\M) \geq \frac{1}{\pi} \, \sys(\M) \, (\sys(\M) + h(\M)).
\end{equation}
\end{proposition}

We present examples showing this result is optimal at the end of this section, \cf~Example~\ref{ex:extremal-wide}.

\begin{proof}
Consider 
$$
\mathbb{U}:=\{x\in\M\mid d(x,\partial\M)\leq\frac{\lambda-1}{2} \sys(\M) \}.
$$
Slightly perturbing the distance function $d(.,\partial\M)$ if necessary, we can assume that this distance function is a Morse function on~$\M$ for which $\frac{\lambda-1}{2} \, \sys(\M)$ is a regular value. 
In this case, $\mathbb{U}$ is a surface with boundary. 
If $\M$ has some ``big bumps", the surface $\mathbb{U}$ may possibly have some holes. More precisely, the surface $\mathbb{U}$ may not be a topological cylinder as some of its boundary components may bound topological disks in $\M$. 

\medskip

Let $\widehat{\mathbb{U}}$ be the union of $\mathbb{U}$ with the topological disks of $\M$ bounded by the boundary components of $\mathbb{U}$. 
Under this construction, $\mathbb{\widehat{U}}$ is a cylinder one of whose boundary components agrees with $\partial\M$. 
Clearly, the height of~$\mathbb{\widehat{U}}$ is equal to~$\frac{\lambda-1}{2} \, \sys(\M)$.
Furthermore, since the inclusion $\mathbb{\widehat{U}} \subset \M$ induces a $\pi_1$-isomorphism, we have $\sys(\mathbb{\widehat{U}}) \geq \sys(\M)$.
Applying Proposition~\ref{finslercylindre} to the cylinder~$\mathbb{\widehat{U}}$ yields
\begin{equation} \label{eq:Uhat}
\vol_{HT}(\mathbb{\widehat{U}})\geq \frac{\lambda-1}{\pi} \, \sys(\M)^2.
\end{equation}

Now, consider the Finsler Mobius band $\M_-:=\M\setminus \mathbb{\widehat{U}}$.

\begin{lemma}\label{case2h}
The height and systole of~$\M_-$ satisfy
\[
h(\M_-)= \sys(\M)
\quad
\mbox{ and } 
\quad
\sys(\M_-)\geq \sys(\M).
\]
\end{lemma}

\begin{proof}
Let $\gamma_-$ be a height arc of~$\M_-$, \cf~Definition~\ref{heightdefinition}.
By construction, $\partial \mathbb{\widehat{U}}=\partial\M_-\cup\partial\M$ and the points of $\partial\M_-$ are at distance $\frac{\lambda-1}{2} \, \sys(\M)$ from $\partial\M$. 
Therefore, the two endpoints of $\gamma_-$ can be connected to~$\partial\M$ by two arcs $\gamma_{1}$ and $\gamma_{2}$ of~$\mathbb{\widehat{U}}$, each of length $\frac{\lambda-1}{2} \, \sys(\M)$. 
Moreover, the arc $\gamma:=\gamma_-\cup\gamma_{1}\cup\gamma_{2}$ with endpoints in~$\partial\M$ induces a nontrivial class in $\pi_{1}(\M,\partial\M)$. 
Therefore, since $h(\M) = \lambda \, \sys(\M)$, we obtain
\begin{align*}
h(\M_-)
&=\ell(\gamma)-(\lambda-1) \, \sys(\M) \\
&\geq h(\M)-(\lambda-1) \, \sys(\M) \\
&\geq \sys(\M).
\end{align*}
Now, let $\gamma$ be a height arc of~$\M$.
By definition, we have $\ell(\gamma)=h(\M)=\lambda \, \sys(\M)$. 
The part $\gamma\cap\mathbb{\widehat{U}}$ of $\gamma$ in $\mathbb{\widehat{U}}$ is made of two arcs, each of length at least $\frac{\lambda-1}{2} \, \sys(\M)$. 
Moreover, the arc $\gamma\cap\M_-$ with endpoints in $\partial\M_-$ induces a nontrivial class in $\pi_1(\M_-,\partial\M_-)$. 
Hence,
\begin{align*}
h(\M_-)
&\leq \ell(\gamma)-\ell(\gamma\cap\mathbb{\widehat{U}})\\
&\leq h(\M)-(\lambda-1) \, \sys(\M)\\
&\leq \sys(\M).
\end{align*}

Since the inclusion $\M_-\subset\M$ induces a $\pi_1$-isomorphism, we obtain
\[
\sys(\M_-)\geq \sys(\M).
\]
\end{proof}

Consider the projective plane~$\RP^2$ defined as the quotient $\M_-/\partial\M_-$, where the boundary $\partial\M_-$ is collapsed to a point. 
Strictly speaking, the Finsler metric on~$\RP^2$ has a singularity at the point to which $\partial\M_-$ collapses, but we can smooth it out. 

\medskip

The following result allows us to derive the systole of~$\RP^2$.

\begin{lemma}\label{metricmobius}
Let $\RP^2$ be the projective plane defined as the quotient $\M/\partial\M$ of a Finsler Mobius band~$\M$.
Then, 
$$
\sys(\RP^2)=\min\left\{h(\M),\sys(\M)\right\}
$$
where $\RP^2$ is endowed with the quotient metric.
\end{lemma}

\begin{proof}
Let $\gamma$ be a noncontractible loop in $\RP^2$. 
The curve~$\gamma$ lifts either to a noncontractible loop in $\M$ or to a noncontractible arc in $\M$ joining two points of the boundary $\partial\M$. In the former case, the length of $\gamma$ is at least $\sys(\M)$, while in the latter, it is at least $h(\M)$.
On the other hand, we can easy construct noncontractible loops in~$\RP^2$ of length $\sys(\M)$ or~$h(\M)$. 
\end{proof}

From Lemma~\ref{metricmobius} and Lemma~\ref{case2h}, the systole of $\RP^2$ is equal to~$\sys(\M)$.
Applying Theorem \ref{Iv02} to $\RP^2$, we obtain 
$$
\vol_{HT}(\M_-) =\vol_{HT}(\RP^2) \geq \frac{2}{\pi} \, \sys^2(\M).
$$
This inequality combined with~\eqref{eq:Uhat} yields
\begin{align*}
\vol_{HT}(\M)
&= \vol_{HT}(\mathbb{\widehat{U}})+\vol_{HT}(\M_-)\\
&\geq \frac{1+\lambda}{\pi} \, \sys(\M)^2.
\end{align*}
Hence the result.
\end{proof}

We conclude this section by describing extremal and almost extremal Finsler metrics when $\lambda \geq 1$.

\begin{example} \label{ex:extremal-wide}
Let $\lambda\in [1,+\infty)$.
\begin{itemize}
\item[(E.1)] 
The horizontal translation $\tau$ of vector $\pi \, \vec{e}_x$ is an isometry of the plane $\mathbb{R}^2$ endowed with the sup-norm. 
The quotient of the strip $\mathbb{R}\times[0,(\lambda-1)\frac{\pi}{2}]$ by the isometry group $\langle\tau \rangle$ generated by $\tau$ is a cylinder~$C$. 
The Finsler mobius band $\M$ obtained by gluing a boundary component of~$C$ to~$\M_{F_{a}}$ along $\partial\M_{F_{a}}$ with $a=\frac{\pi}{3}$, \cf~Section~\ref{sec:candidates}, satisfies $\vol_{HT}(\M)=(1+\lambda)\pi$, $\sys(\M)=\pi$ and $h(\M)=\lambda\pi$. 
See Figure~\ref{fig:extmobius}.(A).

\medskip

\item[(E.2)] 
Endow the plane $\mathbb{R}^2$ with the sup-norm. 
The quotient of the strip $\mathbb{R}\times[-\frac{\pi}{2},\frac{\pi}{2}]$ by the group generated by the map $(x,y)\mapsto (x+\pi,-y)$ is a Finsler Mobius band $\M_\pi$ with $\vol_{HT}(\M_\pi)=2 \pi$, $\sys(\M_\pi)=\pi$ and $h(\M_\pi)= \pi$.
Let $C$ be the Finsler cylinder defined in~(E.1). 
Attach~$C$ to $\mathbb{M_{\pi}}$ via a cylindrical part of arbitrarily small area, \cf~Figure~\ref{fig:extmobius}.(B), so that the resulting space is a Finsler Mobius band $\M$ with \mbox{$\sys(\M)=\pi$}, $h(\M)=\nu_{1}\lambda\pi$ and $\vol_{HT}(\M)=\nu_{2}(\lambda+1)\pi$, where $\nu_{1}, \nu_{2} > 1$ are arbitrarily close to $1$. 
This Finsler Mobius band is almost extremal for the inequality~\eqref{eq:wide} when $\lambda \geq 1$.
\end{itemize}
\end{example}

\begin{figure}[htbp]    
\begin{subfigure}[!]{0.4\textwidth}
\includegraphics[width=\textwidth]{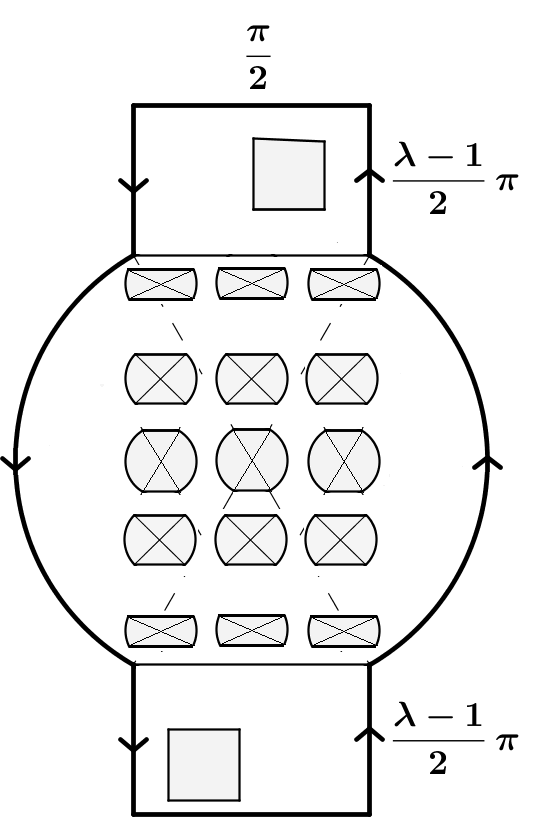}
\caption{An extremal metric}
\end{subfigure}%
         \qquad\qquad\qquad
\begin{subfigure}[!]{0.4\textwidth}
\includegraphics[width=\textwidth]{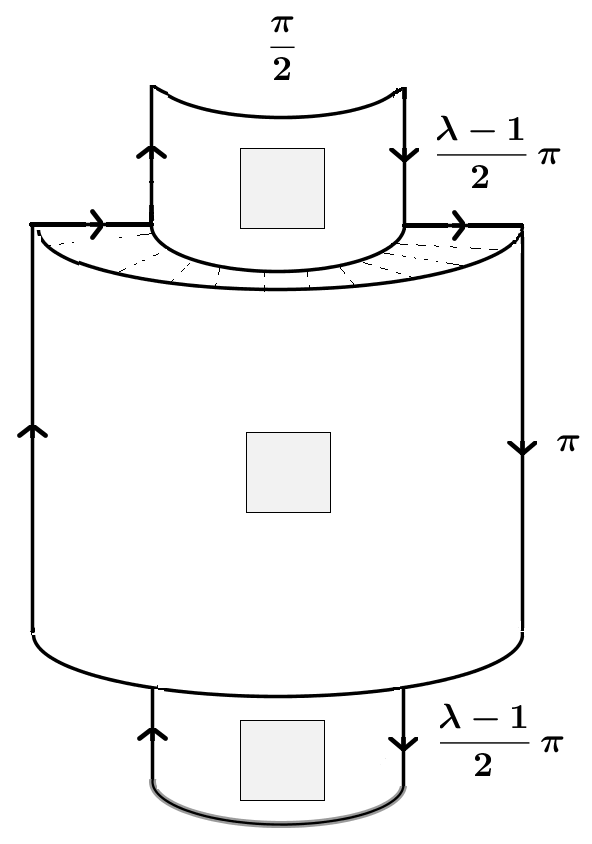}
\caption{An almost extremal metric} \label{almostextremal}
\end{subfigure}
\caption{Extremal and almost extremal Finsler Mobius bands when $h(\M)\geq \sys(\M)$.}\label{fig:extmobius}
\end{figure}

\section{Systolic inequalities on narrow Finsler Mobius bands}  \label{sec:narrow}
In this section, we give a proof of Theorem~\ref{finslermobius} for narrow Finsler Mobius bands, that is, when $\lambda <1$. 
More precisely, we prove the following result.

\begin{proposition} \label{prop:narrow}
Let $\M$ be a Finsler Mobius band with $h(\M) < \sys(\M)$.
Then
\[
\vol_{HT}(\M) \geq \frac{2}{\pi} \, \sys(\M) \, h(\M).
\]
\end{proposition}

This inequality is optimal.
Extremal Finsler metrics can be constructed as follows.

\begin{example} \label{ex:extremal-narrow}
Let $\lambda < 1$.
Endow the plane $\mathbb{R}^2$ with the sup-norm. 
The quotient of the strip $\mathbb{R}\times[-\frac{\pi}{2}\lambda,\frac{\pi}{2}\lambda]$ by the group generated by the map $(x,y)\mapsto (x+\pi,-y)$ is a Finsler Mobius band $\M$ with $\vol_{HT}(\M)=2\lambda\pi$, $\sys(\M)=\pi$ and $h(\M)=\lambda\pi$.
\end{example}

Before proceeding to the proof of this proposition, we need to introduce a few definitions and notions.

\begin{definition} \label{def:K}
The orientable double cover of a Finsler Mobius band~$\M$ is a cylinder denoted by~$C$.
The points of~$C$ which are at the same distance from each boundary component of~$C$ form a simple closed curve~$c$ invariant under deck transformations.
The \emph{soul} of~$\M$, denoted by~$\sigma$, is defined as the projection of~$c$ to~$\M$.
Note that $\sigma$ is a nonorientable simple closed curve of~$\M$ whose homotopy class generates~$\pi_1(\M)$.

\medskip

Let $\K$ be the Finsler Klein bottle obtained by attaching to~$\M$ another copy~$\M'$ of~$\M$ along their boundary.
(The Finsler metric on~$\K$ may have a singularity line along~$\partial \M$.)
The isometry of~$\K$ switching the souls of $\M$ and~$\M'$, and leaving~$\partial \M$ pointwise fixed is called the \emph{soul-switching symmetry} of~$\K$.

\medskip

Let $s,h \in \R_+$.
The \emph{systole-height inequality} on the Mobius band is said to be satisfied for~$(s,h)$ if for every Finsler Mobius band~$\M$ with $\sys(\M) \geq s$ and $h(\M) \geq h$, we have
\[
\vol_{HT}(\M) \geq \frac{2}{\pi} \, sh.
\]
By scale invariance, if the systole-height inequality is satisfied for~$(s,h)$, then it is also satisfied for $(s',h')$ with $\frac{h'}{s'}=\frac{h}{s}$.
\end{definition}

We first prove the following preliminary result.

\begin{lemma} \label{lem:1/2}
Let $\lambda \in (0,1]$.
Suppose that the systole-height inequality on the Mobius band is satisfied for~$(s,h)$ with~$\frac{h}{s}=\lambda$.
Then, it is also satisfied for~$(s,h)$ with~$\frac{h}{s}=\frac{\lambda}{2}$.
\end{lemma}

\begin{proof}
Let $\M$ be a Finsler Mobius band with $\sys(\M) \geq s$ and $h(\M) \geq h$, where $\frac{s}{h}=\frac{\lambda}{2}$.
Consider the Klein bottle~$\K$ made of two copies of~$\M$ defined in Definition~\ref{def:K} and cut it open along the soul~$\sigma'$ of~$\M'$.
The resulting surface is a Finsler Mobius band denoted by~$2\M$ whose boundary component double covers~$\sigma'$ in~$\K$.

\medskip

Let $\alpha$ be a noncontractible loop of~$2\M$.
Decompose $\alpha$ into two parts $a=\alpha \cap \M$ and $a'=\alpha \cap \M'$ with $\alpha = a \cup a'$.
The parts $a$ and~$a'$ form two collections of arcs with endpoints lying in~$\partial \M = \partial \M'$.
By construction, the image~$a''$ of~$a'$ by the soul-switching symmetry lies in~$\M$.
Furthermore, the union $\bar{\alpha} = a \cup a''$ forms a closed curve lying in~$\M$ and homotopic to~$\alpha$ in~$2\M$ (and so noncontractible in~$\M$).
Since $\bar{\alpha}$ has the same length as~$\alpha$, we conclude that $\sys(2\M) \geq \sys(\M) \geq s$.

Actually, since the inclusion $\M \subset 2\M$ is a strong deformation retract, we derive the relation $\sys(2\M) = \sys(\M)$.
But we will not make use of this equality in the sequel.

\medskip

By construction, the distance between the soul~$\sigma$ and~$\partial \M$ (and between~$\sigma'$ and~$\partial \M'$) is at least~$\tfrac{1}{2} h(\M)$.
This implies that $h(2\M) \geq 2 h(\M) \geq 2h$.

Actually, we can show that $h(2\M) = 2 h(\M)$ (but we will not make use of this relation afterwards).
Indeed, let $\alpha$ be a height arc of~$\M$.
By definition, $\ell(\alpha) = h(\M)$.
Denote by~$\alpha'$ its image in~$\M'$ by the soul-switching symmetry of~$\K$.
The trace of the union~$\alpha \cup \alpha'$ to~$2\M$ defines an arc with endpoints in~$\partial (2\M)$ inducing a nontrivial class in~$\pi_1(2\M,\partial(2\M))$.
The length of this arc is twice the length of~$\alpha$.
Therefore, $h(2\M) \leq 2 h(\M)$. 
Hence, the equality $h(2\M) = 2 h(\M)$. 

\medskip

In conclusion, the Mobius band~$2\M$ satisfies $\sys(2\M) \geq s$ and $h(2\M) \geq 2h$.
Since $\frac{2h}{s} = \lambda$, the systole-height inequality is satisfied for~$(s,2h)$ by the lemma assumption.
Therefore, 
\[
2 \, \vol_{HT}(\M) = \vol_{HT}(2 \M) \geq \frac{4}{\pi} \, sh
\]
and the result follows.
\end{proof}

We establish a second preliminary result.

\begin{lemma}\label{lem:(x+y)/2}
Let $\lambda_1, \lambda_2 \in \R$ such that $0< \lambda_1 < \lambda_2 \leq 1$. 
Suppose that the systole-height inequality on the Mobius band is satisfied for~$(s,h)$ with $\frac{h}{s} = \lambda_1$ or~$\lambda_2$.
Then, it is also satisfied for~$(s,h)$ with $\frac{h}{s}=\frac{\lambda_1+\lambda_2}{2}$.
\end{lemma}

\begin{proof}
Let $\M$ be a Finsler Mobius band with $\sys(\M) \geq s$ and $h(\M) \geq h$, where $\frac{h}{s}=\frac{\lambda_1+\lambda_2}{2}$.
Consider the Klein bottle~$\K$ made of two isometric Mobius bands $\M$ and~$\M'$ with souls $\sigma$ and~$\sigma'$ defined in Definition~\ref{def:K}.
Consider also
\[
\M_1 = \{ x \in \K \mid \lambda_2 \, d(x,\sigma) \leq \lambda_1 \, d(x,\sigma') \}
\]
and 
\[
\M_2 = \{ x \in \K \mid \lambda_2 \, d(x,\sigma) \geq \lambda_1 \, d(x,\sigma') \}
\]
Note that if we drop the multiplicative constants $\lambda_1$ and~$\lambda_2$ in the definitions of~$\M_1$ and~$\M_2$, we obtain $\M$ and~$\M'$.

\medskip

The subset $\M_1$ is a Finsler Mobius band contained in~$\M$.
Similarly, the subset $\M_2$ is a Finsler Mobius band containing~$\M'$.
Observe also that the Mobius bands $\M_1$ and~$\M_2$ cover~$\K$ and that their interiors are disjoint.

Every point $z \in \partial \M_i$ satisfies the equality
\begin{equation} \label{eq:21}
\lambda_2 \, d(z,\sigma_1) = \lambda_1 \, d(z,\sigma_2).
\end{equation}
By symmetry, the distance between~$\sigma_1=\sigma$ and~$\sigma_2=\sigma'$ is equal to~$h(\M)$.
It follows by the triangle inequality that 
\begin{equation} \label{eq:12}
d(z,\sigma_1) + d(z,\sigma_2) \geq h(\M).
\end{equation}
As a result of the relations \eqref{eq:21} and~\eqref{eq:12}, we obtain
\begin{equation} \label{eq:zi}
d(z,\sigma_i) \geq \frac{\lambda_i}{\lambda_1 + \lambda_2} \, h(\M) .
\end{equation}

Now, let $\alpha$ be an arc of~$\M_i$ with endpoints $x,y \in \partial \M_i$ inducing a nontrivial class in~$\pi_1(\M_i,\partial \M_i)$.
As $\alpha$ intersects~$\sigma_i$, we deduce from \eqref{eq:zi} that
\[
\ell(\alpha) \geq d(x,\sigma_i) + d(y,\sigma_i) 
 \geq \frac{2 \lambda_i}{\lambda_1+\lambda_2} h(\M).
\]
Therefore, 
\begin{equation} \label{eq:hi}
h(\M_i) \geq \frac{2 \lambda_i}{\lambda_1+\lambda_2} h(\M) \geq \frac{2 \lambda_i}{\lambda_1+\lambda_2} h.
\end{equation}

\medskip

In another direction, we can also bound from below the systole of~$\M_1$ and~$\M_2$ as follows.

For the systole of~$\M_1$, since the inclusion $\M_1 \subset \M$ induces a $\pi_1$-isomorphism, we derive 
\begin{equation} \label{eq:sys1}
\sys(\M_1) \geq \sys(\M) \geq s.
\end{equation}
Note that the first inequality may be strict as the inclusion $\M_1 \subset \M$ is strict.

For the systole of~$\M_2$, we argue as in Lemma~\ref{lem:1/2}.
Let $\alpha$ be a noncontractible loop of~$\M_2$.
Decompose $\alpha$ into two parts $a=\alpha \cap \M$ and $a'=\alpha \cap \M'$ with $\alpha=a \cup a'$.
The union $\bar{\alpha}=a^* \cup a'$, where $a^*$ is the image of~$a$ by the soul-switching symmetry of~$\K$, forms a closed curve of length~$\ell(\alpha)$ lying in~$\M'$ and homotopic to~$\alpha$ in~$\M_2$.
Hence,
\begin{equation} \label{eq:sys2}
\sys(\M_2) \geq \sys(\M) \geq s.
\end{equation}

\medskip

The systole-height inequality on the Mobius band is satisfied for~$(s,\frac{2 \lambda_i}{\lambda_1+\lambda_2} h)$ from the lemma assumption since $\frac{2 \lambda_i}{\lambda_1 + \lambda_2} \frac{h}{s} = \lambda_i$.
From the bounds~\eqref{eq:hi}, \eqref{eq:sys1} and~\eqref{eq:sys2}, this inequality applies to~$\M_i$ and yields
\begin{equation} \label{eq:lambdai}
\vol_{HT}(\M_i) \geq \frac{4}{\pi} \, \frac{\lambda_i}{\lambda_1+\lambda_2} \, sh.
\end{equation}

Finally, recall that the Mobius bands $\M_1$ and~$\M_2$ cover~$\K$ and that their interiors are disjoint.
By adding up~\eqref{eq:lambdai} for $i=1,2$, we conclude that
\begin{align*}
2 \, \vol_{HT}(\M) & = \vol_{HT}(\K) \\
& = \vol_{HT}(\M_1) + \vol_{HT}(\M_2) \geq \frac{4}{\pi} \, sh
\end{align*}
Hence the result.
\end{proof}

\begin{remark}
At first glance, it seems more natural to assume that $\sys(\M)=s$ and $h(\M)=h$ in the definition of the systole-height inequality, \cf~Definition~\ref{def:K}.
Observe that the proof of Lemma~\ref{lem:1/2} carries over with this alternative notion.
However, we have not been able to directly prove a result similar to Lemma~\ref{lem:(x+y)/2} with this more restrictive notion.
The reason is that the inequality~\eqref{eq:sys1}, namely $\sys(\M_1) \geq \sys(\M)$, may be strict as the inclusion $\M_1 \subset \M$ is strict.
To get around this subtle difficulty, we relaxed the original definition and formulated the systole-height inequality in terms of lower bounds for the systole and the height of the Mobius band.
%
\end{remark}

We can now proceed to the proof of Proposition~\ref{prop:narrow}.

\begin{proof}[Proof of Proposition~\ref{prop:narrow}]
By Lemma \ref{lem:1/2}, for every nonnegative integer~$k$, the systole-height inequality on the Mobius band is satisfied for~$(s,h)$ with $\frac{h}{s} = \frac{1}{2^k}$. 
Combined with Lemma~\ref{lem:(x+y)/2}, this implies that the systole-height inequality is satisfied for every~$(s,h)$ where $\frac{h}{s}$ is a dyadic rational of~$(0,1)$.
Since the height, the systole and the volume are continuous over Finsler metrics, the result follows from the density of the dyadic rationals in~$[0,1]$.
\end{proof}

\forget
\begin{remark}
(The inequality $\sys(\M_1) \geq \sys(\M)$ may be strict as the inclusion $\M_1 \subset \M' \simeq \M$ is strict.
This is the reason why we formulated the $(s,h)$-systole-height inequality in terms of lower bounds for the systole and the height.)

Similarly, the subset $\M_2$ is a Finsler Mobius band containing in~$\M$ with the same soul~$\sigma$ as~$\M$.
Furthermore, $\sys(\M_2) \geq \sys(M) \geq s$ (in this case, we can show that $\sys(M_2) = \sys(\M)$ as in Lemma~\ref{lem:1/2}) and $h(\M_1) = \frac{2 \lambda_2}{\lambda_1+\lambda_2} h(\M) \geq \frac{2 \lambda_2}{\lambda_1+\lambda_2} h$.
\end{remark}

We can now proceed to the proof of Proposition~\ref{prop:narrow}.

\begin{proof}[Proof of Proposition~\ref{prop:narrow}]

In particular, for every nonnegative integer~$k$, the systole-height inequality on the Mobius band is satisfied for~$(s,h)$ with $\frac{h}{s} = \frac{1}{2^k}$.

\end{proof}
\forgotten

\section{Systolic inequality on Finsler Klein bottles} \label{sec:klein}
In this section, we show that the systolic area of Finsler Klein bottles with soul, soul-switching or rotational symmetries is equal to $\frac{2}{\pi}$. 

\begin{definition} \label{def:sym}
Recall that every Riemannian Klein bottle is conformally equivalent to the quotient of~$\mathbb{R}^2$ by the isometry group~$G$ generated by the glide reflection $(x,y)\mapsto (x+\pi,-y)$ and the vertical translation $(x,y)\mapsto (x,y+2b)$ with $b>0$.

The flat Klein bottle~$\K=\mathbb{R}^2/G$ decomposes into two Mobius bands whose souls correspond to the projections of the lines $\{y=0\}$ and~$\{y=b\}$.
The boundary of the two Mobius bands agrees with the projection of the line $\{y=\frac{b}{2} \}$ (or $\{y=-\frac{b}{2}$).
A Finsler metric on $\K$ has a \emph{soul symmetry} if its lift to $\mathbb{R}^2$ is invariant by the map $(x,y)\mapsto(x,-y)$.
Similarly, a Finsler metric on $\K$ has a \emph{soul-switching symmetry} if its lift to $\mathbb{R}^2$ is invariant by the map $(x,y)\mapsto(x,b-y)$.
Finally, a Finsler metric on $\K$ has a \emph{rotational symmetry} if its lift to $\mathbb{R}^2$ is invariant by the map $(x,y)\mapsto(x+\theta,y)$ for every $\theta\in[0,2\pi]$.

These definitions are consistent with the notions introduced in~\ref{sec:narrow}.
\end{definition} 

In 1986, C. Bavard established an optimal isosystolic inequality for Riemannian Klein bottles, \cf~\cite{Ba86}. 
Alternative proofs can be found in \cite{Sak88, Ba88, Ba06}. 
All the proofs are based on the uniformization theorem. 
In fact, as mentioned in the introduction, the problem boils down to consider Riemannian Klein bottles invariant under soul (and rotational) symmetry. 
These Klein bottles are made of two isometric copies of Riemannian Mobius bands. 
Thus, in the end, the systolic inequality on Riemannian Klein bottles follows from optimal systolic inequalities on Riemannian Mobius bands, \cf~\cite{Pu52, Bl62}. \\

There is no known optimal isosystolic inequality on the Klein bottle for Finsler metrics. However, we obtain the following partial result similar to the Riemannian case.
Note that in the Riemannian case, the hypothesis is automatically satisfied by an average argument.

\begin{theorem}\label{finslerkleinbottle2}
Let $\K$ be a Finsler Klein bottle with a soul, soul-switching or rotational symmetry. Then
\begin{equation} \label{k2}
\frac{\vol_{HT}(\K)}{\sys^2(\K)}\geq \frac{2}{\pi}.
\end{equation}
Moreover, the inequality is optimal.
\end{theorem}

Let $\K=\mathbb{R}^2/G$ be a Finsler Klein bottle (with or without symmetry). Denote by $\M$ the Finsler Mobius band obtained by cutting $\K$ open along the soul given by the projection of the line $\left\{y=b\right\}$ to $\K$. 
The proof of the inequality~\eqref{k2} in Theorem \ref{finslerkleinbottle2} follows by combining the next three lemmas.

\begin{lemma}\label{kcase1}
If $h(\M)\geq \sys(\K)$ then the inequality~\eqref{k2} holds true.
\end{lemma}

\begin{proof}
The inclusion $\M\subset\K$ induces a $\pi_1$-isomorphism. Hence, $\sys(\K)\leq \sys(\M)$.
Now, we only have two cases to consider.

First, if $h(\M)\leq \sys(\M)$, then we deduce from Theorem \ref{finslermobius}, first case, that
\begin{align*}
\frac{\vol_{HT}(\K)}{\sys^{2}(\K)}
&\geq\frac{\vol_{HT}(\M)}{\sys(\M) \, h(\M)}\\
&\geq\frac{2}{\pi}.
\end{align*}

Second, if $h(\M)\geq \sys(\M)$, then we deduce from Theorem~\ref{finslermobius}, second case, that
\begin{align*}
\frac{\vol_{HT}(\K)}{\sys^{2}(\K)}
&\geq\frac{1}{\pi}\left(1+\frac{\sys(\M)}{h(\M)}\right)\frac{\sys(\M) \, h(\M)}{\sys^2(\K)}\\
&\geq \frac{1}{\pi}\left(\frac{h(\M)}{\sys(\K)}+\frac{\sys(\M)}{\sys(\K)}\right)\frac{\sys(\M)}{\sys(\K)}\\
&\geq\frac{2}{\pi},
\end{align*}
since both $\frac{h(\M)}{\sys(\K)}$ and $\frac{\sys(\M)}{\sys(\K)}$ are greater or equal to~$1$.
\end{proof}

The next three lemmas show that the assumption of Lemma~\ref{kcase1} is satisfied when the Finsler metric on~$\K$ has soul, soul-switching or rotational symmetries.

\begin{lemma}\label{kcase2}
If $\K$ is a Finsler Klein bottle with a soul symmetry then $h(\M)\geq \sys(\K)$.
\end{lemma}

\begin{proof}
Observe that the soul symmetry of $\K$ leaves both $\M$ and $\partial\M$ invariant. 
Given an arc $\alpha$ of $\M$, we denote by $\alpha^*$ the arc of~$\M$ symmetric to~$\alpha$ by the soul symmetry. 
Let $\gamma$ be a height arc of~$\M$, \cf~Definition~\ref{heightdefinition}.
This arc decomposes into two subarcs $\alpha$ and~$\beta$ connecting $\partial\M$ to the soul of~$\M$ with $\ell(\alpha)\leq \ell(\beta)$. 
The arc $\alpha\cup\alpha^*$ with endpoints in~$\partial\M$ induces a nontrivial class in $\pi_1(\M,\partial\M)$ of length at most the length of $\gamma=\alpha\cup\beta$. 
By definition of $h(\M)$, we conclude that $\alpha\cup\alpha^*$ is as long as $\gamma$ and so is length-minimizing in its relative homotopy class.
In particular, it is geodesic. 
Thus, the arc $\gamma$, which has the subarc~$\alpha$ in common with~$\alpha \cup \alpha^*$, agrees with $\alpha\cup\alpha^*$.
In particular, it is invariant by the soul symmetry. 
As a result, the arc $\gamma$ induces a noncontractible loop on $\K$ after identification of the points of $\partial\M$ under the soul symmetry.
Hence, $\ell(\alpha) \geq \sys(\K)$.
\end{proof}

\begin{lemma}\label{kcase4}
If $\K$ is a Finsler Klein bottle with a soul-switching symmetry then $h(\M)\geq \sys(\K)$.
\end{lemma}

\begin{proof}
Since the Finsler Klein bottle $\K$ has a soul-switching symmetry, we can assume that it is composed of two isometric Finsler Mobius bands $\M_1$ and $\M_2$. By symmetry, we have $h(\M)=2h(\M_1)$. Given an arc $\alpha$ of $\M_1$, we denote by $\alpha^*$ the arc of~$\M_2$ symmetric to~$\alpha$ by the soul-switching symmetry. 
Let $\gamma$ be a height arc of~$\M_1$.
This arc decomposes into two subarcs $\alpha_1$ and~$\beta_1$ connecting $\partial\M_1$ to the soul of~$\M_1$. 
The arc $\eta=:\alpha_1\cup\alpha_1^*\cup\beta_1\cup\beta_1^*$, \ie, the union of the arc $\gamma$ and its symmetric image $\gamma^*$, has endpoints in~$\partial\M$ and induces a nontrivial class in $\pi_1(\M,\partial\M)$ of length equal to $2h(\M_1)$. 
Moreover, this arc $\eta$ induces a noncontractible loop on $\K$ after identification of the points of $\partial\M$ under the soul-switching symmetry. 
We conclude that $h(\M)\geq \sys(\K)$.
 \end{proof}

\begin{lemma}\label{kcase3}
If $\K$ is a Finsler Klein bottle with rotational symmetry then $h(\M)\geq \sys(\K)$.
\end{lemma}

\begin{proof}
Observe that the rotational symmetries of $\K$ leave both $\M$ and $\partial\M$ invariant. 
Given an arc $\alpha$ of $\M$, we denote by $\alpha^\theta$ the arc of~$\M$ symmetric to~$\alpha$ by the rotational symmetry of angle~$\theta$. 
Let $\gamma$ be a length-minimizing arc of~$\M$ parametrized (proportionally to its length) by~$[0,1]$ with endpoints in~$\partial \M$ inducing a nontrivial class in $\pi_1(\M,\partial\M)$.
Note that $\gamma$ is a geodesic arc of length~$h(\M)$.
By the first variation formula for Finsler metrics, \cf~\cite{Shen}, the geodesic arc~$\gamma$ is perpendicular to $\partial\M$.
It follows that the endpoints $\gamma(0)$ and $\gamma(1)$ of~$\gamma$ in~$\partial \M$ are distinct.
Since the Finsler metric is invariant under rotational symmetry, there exists $\theta\in (0,2\pi)$ such that $\gamma^{\theta}(0)=\gamma(1)$. 
Both symmetric arcs~$\gamma$ and $\gamma^{\theta}$ are perpendicular to $\partial\M$.
In particular, their tangent vectors at $\gamma^{\theta}(0)=\gamma(1)$ coincide up to sign. 
Therefore, the geodesic arcs $\gamma$ and~$\gamma^\theta$ agree up to reparametrization.
More precisely, $\gamma^\theta(s) = \gamma(1-s)$ for every $s \in [0,1]$.
Thus, $\gamma^{2\theta}=(\gamma^\theta)^\theta=\gamma$.
Hence, $2\theta = 2 \pi$, that is, $\theta = \pi$.
Therefore, the arc~$\gamma$ projects to a closed curve in~$\K$.
It follows that the length of $\gamma$ is at least $\sys(\K)$.
\end{proof}

\begin{example} \label{ex:extremal2}
The following two examples show that the inequality~\eqref{k2} is optimal.
\begin{enumerate}
\item The quotient of $\mathbb{R}^2$, endowed with the sup-norm, by the isometry group~$G$ generated by the glide reflection with parameter~$b=\frac{\pi}{2}$, \cf~Definition~\ref{def:sym}, is a Finsler Klein bottle with soul, soul-switching and rotational symmetries, of area~$2\pi$ and systole~$\pi$.
\item The Finsler Mobius band $\M_{F_{\frac{\pi}{3}}}$, \cf~Section~\ref{sec:candidates}, whose opposite boundary points are pairwise identified defines a Finsler Klein bottle with soul, soul-switching and rotational symmetries, of area~$2\pi$ and systole~$\pi$. 
\end{enumerate}
\end{example}

We believe that Theorem~\ref{finslerkleinbottle2} holds true for every Finsler Klein bottle (not necessarily invariant under soul, soul-switching or rotational symmetries).
More precisely, we state the following conjecture.

\begin{conjecture}
Let $\K$ be a Finsler Klein bottle. 
Then 
$$
\frac{\vol_{HT}(\K)}{\sys^2(\K)}\geq \frac{2}{\pi}.
$$
Moreover, the inequality is optimal.
\end{conjecture}

\begin{remark}
If the conjecture is true, the Finsler systolic areas of $\RP^2$, $\T^2$ and $\K$ would be the same.
\end{remark}

\section{A non-optimal systolic inequality on Finsler Klein bottles} \label{sec:kleinfail}

In this section, we present a non-optimal systolic inequality on Finsler Klein bottles.

\begin{proposition} \label{prop:nonsharp}
Let $\K$ be a Finsler Klein bottle. 
Then 
$$
\frac{\vol_{HT}(\K)}{\sys^2(\K)}\geq \frac{\sqrt{2}}{\pi}.
$$
\end{proposition}

\begin{proof}
Every symmetric convex body $C\subset \mathbb R^n$ admits a unique ellipsoid~$E(C)$ of maximal volume among the ellipsoids contained in $C$.
This ellipsoid, called John's ellipsoid, continuously varies with~$C$ for the Hausdorff topology.
Furthermore, it satisfies the double inclusion, \cf~\cite[Corollary~11.2]{grub} for instance,
\begin{equation} \label{eq:loewner}
E(C) \subset C \subset \sqrt{n} \, E(C).
\end{equation}

Given a Klein bottle~$\K$ with a Finsler metric~$F$, we define a continuous Riemannian metric~$g$ on~$\K$ by replacing the Minkowski norm~$F_x$ on each tangent space~$T_x\K$ by the inner product induced by the John ellipsoid~$E(B_x)$, where $B_x$ is the unit ball of~$F_x$.
The double inclusion~\eqref{eq:loewner} satisfied by~$E(B_x)$ implies that $\frac{1}{\sqrt{2}} \, \sqrt{g} \leq F \leq \sqrt{g}$.
Hence, 
\begin{equation} \label{eq:vol}
\sys(F) \leq \sys(g)  \quad \mbox{ and } \quad \frac{1}{2} \, \vol(g) \leq \vol_{HT}(F) 
\end{equation} 
From the optimal Riemannian systolic inequality on the Klein bottle~\cite{Ba86}, we obtain
\[
\vol_{HT}(F) \geq \frac{1}{2} \, \vol(g) \geq \frac{1}{2} \, \frac{2\sqrt{2}}{\pi} \, \sys^2(g) \geq \frac{\sqrt{2}}{\pi} \, \sys^2(F).
\]  
\end{proof}

\begin{remark}
The naive volume bound in~\eqref{eq:vol} can be improved into $\vol(g) \leq \frac{\pi}{2} \, \vol_{HT}(F)$, see the proof of~\cite[Theorem~4.11]{ABT}.
This leads to the better lower bound~$\frac{4\sqrt{2}}{\pi^2}$ in Proposition~\ref{prop:nonsharp}.
\end{remark}

\bibliographystyle{alpha}

\end{document}